\documentclass[a4paper,10pt]{amsart}
\usepackage[english]{babel}
\usepackage[utf8]{inputenc}
\usepackage[T1]{fontenc}
\usepackage{csquotes}
\usepackage[style=numeric,
	useprefix,%
	giveninits=true,%
	hyperref,%
	doi=false,%
	url=false,%
	isbn=false,%
	backend=bibtex,%
	maxbibnames=99%
	]{biblatex}
\bibliography{./BIB-BernsteinW11}

\usepackage{amssymb}
\usepackage{mathrsfs}
\usepackage{hyperref}
\usepackage[usenames,dvipsnames]{xcolor}
\hypersetup{colorlinks,%
citecolor=Black,%
filecolor=Black,%
linkcolor=Black,%
urlcolor=Black}
\usepackage{enumitem}	

\newcommand{\scr}[1]{\mathscr{#1}}

\newcommand{\bb}[1]{\mathbb{#1}}
\newcommand{\cal}[1]{\mathcal{#1}}
\newcommand{\N}{\mathbb{N}}	
\newcommand{\Z}{\mathbb{Z}}	
\newcommand{\R}{\mathbb{R}}	
\newcommand{\dd}{\,\mathrm{d}}	
\newcommand{\de}{\partial}		
\newcommand{\THEN}{\Rightarrow}	

\newcommand{\HH}{\bb H}

\newcommand{\spt}{\mathrm{spt}}
\newcommand{\loc}{\mathrm{loc}}

\newcommand{\grad}{\nabla}
\newcommand{\LL}{\mathcal L}
\newcommand{\did}{\,\mathrm{d}}  

\usepackage{dsfont}

   \def\XXint#1#2#3{{\setbox0=\hbox{$#1{#2#3}{\int}$}
        \vcenter{\hbox{$#2#3$}}\kern-.5\wd0}}

\theoremstyle{plain}
\newtheorem{proposition}{Proposition}[section]
\newtheorem{theorem}[proposition]{Theorem}
\newtheorem{lemma}[proposition]{Lemma}
\newtheorem{corollary}[proposition]{Corollary}
\newtheorem{thm}{Theorem}[section]

\theoremstyle{definition}
\newtheorem{definition}[proposition]{Definition}
\newtheorem{remark}[proposition]{Remark}

\theoremstyle{remark}

\title[
Sobolev Bernstein problem in the Heisenberg group]{
The Bernstein problem for Sobolev intrinsic graphs in the Heisenberg group
}

\author[Nicolussi Golo]{Sebastiano Nicolussi Golo}
\address[Nicolussi Golo]{Department of Mathematics, University of Fribourg, Ch. du Musée 23, 1700 Fribourg, Orcid ID: https://orcid.org/0000-0002-3773-6471}
\email{sebastiano2.72@gmail.com}

\author[Serra Cassano]{Francesco Serra Cassano}
\address[Serra Cassano]{Dipartimento di Matematica, Università di Trento, Via Sommarive 14, 38123 Trento, Italy}
\email{francesco.serracassano@unitn.it}

\author[Vedovato]{Mattia Vedovato}
\address[Vedovato]{Würth Srl}
\email{vedovato.mattia@gmail.com}

\subjclass[2010]{%
	53C17,
	49Q20
	}
\keywords{%
	Sub-Riemannian Geometry, %
	Sub-Riemannian Perimeter, %
	Heisenberg Group, %
	Bernstein Problem%
	}

\date{\today}

\begin{document}
\maketitle

\begin{abstract}
	In the first Heisenberg group,
	we study entire, locally Sobolev intrinsic graphs 
	that are stable for the sub-Riemannian area.
	We show that, under appropriate integrability conditions for the derivatives,	
	the intrinsic graph 
	must be an intrinsic plane, i.e., a coset of a two dimensional subgroup.
	This result extends \cite{MR3984100} beyond the Lipschitz class.
\end{abstract}

\setcounter{tocdepth}{1}
\phantomsection
\addcontentsline{toc}{section}{Contents}
\tableofcontents


\section{Introduction}

This paper is an improvement of the results from~\cite{MR3984100}:
in the first Heisenberg group $\bb H$,
we extend the Bernstein problem from the Lipschitz to the Sobolev class.
We refer to \cite{MR3984100} for a broader introduction of the subject and to \cite{SerraCassano2020} for a general account of the Bernstein problem in  Heisenberg groups.
For recent results in the study of minimal surfaces in the sub-Riemannian first Heisenberg see \cite{zbMATH07541296,NGR}.
There has been progress also in the sub-Finsler first Heisenberg group,
see~\cite{MR4793141,MR4676403,MR4603626,MR4529384},
and in higher Heisenberg groups, see~\cite{2024arXiv240920359G}.

For $\omega\subset\R^2$ open and $f:\omega\to\R$ of class $W^{1,1}_\loc(\omega)$, 
we consider the area functional  
\[
\scr A_f(K) = \int_K \sqrt{ 1 + (\grad^ff)^2 } \did\LL^2 ,
\]
where $K\subset\omega$ is compact and $\grad^ff = \de_\eta f + f\de_\tau f$,
with $(\eta,\tau)$ coordinates of $\R^2$.
Let us recall that $\scr A_f(\omega)$ represents the sub-Riemannian area of the intrinsic graph induced by $f$; see Section~\ref{sec06152116} below.
We can state the Bernstein problem for intrinsic graphs in the first Heisenberg group
as to characterize $f\in W^{1,1}_\loc(\R^2)$ that minimizes $\scr A_f(K)$ for all $K\Subset \R^2$.
In other words, we look for functions $f\in W^{1,1}_\loc(\R^2)$ so that, for every $K\Subset\R^2$, if $f=g$ on $\R^2\setminus K$, then $\scr A_f(K) \le \scr A_g(K) $.

We say that $f\in W^{1,1}_\loc(\omega)$ is \emph{stationary} if, for all $\phi\in C^\infty_c(\omega)$,
\[
\frac{\dd}{\dd\epsilon}|_{\epsilon=0}\scr A_{f+\epsilon \phi}(\spt(\phi)) = 0 .
\]
A function $f\in W^{1,1}_\loc(\omega)$ is \emph{stable} if it is stationary and, for all $\phi\in C^\infty_c(\omega)$, 
\[
\frac{\dd^2}{\dd\epsilon^2}|_{\epsilon=0}\scr A_{f+\epsilon \phi}(\spt(\phi)) \ge 0 .
\]

In~\cite{MR3984100}, the first two named authors proved that the only stable locally Lipschitz functions $f:\R^2\to\R$ are 
those of the form
\[
f(\eta,\tau) = a\eta + b
\]
for some $a,b\in\R$.
These functions are characterized as those whose intrinsic graph in the Heisenberg group is a vertical plane: 
see Remark~\ref{rem669a6cca} below.

The same conclusion was known for $f$ of class $C^2$ after~\cite{MR2333095}, and for $f$ of class $C^1$ after~\cite{Galli2015}.

In the present paper we weaken the hypothesis on $f$ to be of class $ W^{1,q}_\loc$.
We need $q$ to be larger than 4 and we also need an additional exponential integrability of the derivative $\de_\tau f$.
Under these hypothesis, we can obtain the same conclusion using the strategy of~\cite{MR3984100} and~\cite{MR2333095}, thanks to the existence and regularity of Lagrangian coordinates proved in~\cite{zbMATH07739119}.

Our main result is the following theorem:

\begin{thm}[Bernstein problem with Sobolev regularity]\label{thm669a693f}
	Suppose $f\in W^{1,q}_\loc(\R^2)$ with $q>4$,
	and
	\begin{equation}\label{eq669a34fb}
		\exists \kappa>0\text{ such that }
		\exp(\kappa|\de_\tau f|) \in L^1_\loc(\R^2) .
	\end{equation}
	If $f$ is stable, then there are $a,b\in\R$ such that $f(\eta,\tau) = a\eta + b$,
	for each $(\eta,\tau)\in\R^2$
\end{thm}
The proof of Theorem~\ref{thm669a693f} is at the end of Section~\ref{sec09022317}.

In sections 7 and 8 of \cite{MR3984100}, two examples of stable intrinsic graphs were presented.
Those two examples do not satisfy the hypothesis of Theorem~\ref{thm669a693f} because they are graphs of functions $f\in W^{1,p}_\loc(\R^2)$ with $p=2$ but not $p=4$.
Those examples are now known not to be perimeter minimizers, as shown in \cite[Theorem 1.3]{zbMATH07541296}.

In the proof of Theorem~\ref{thm669a693f}, we show that the intrinsic graph of $f$ is ruled by horizontal lines.
We then prove that the graph of $f$ is a vertical plane using the stability of $f$, not its minimality.
This is different from \cite[Theorem 1.3]{zbMATH07541296}, where it has been proved that as soon as the intrinsic graph $\Gamma_f$ of $f$ is area minimizing and ruled by horizontal straight lines (i.e., isometric embeddings of $\R$), then $\Gamma_f$ is a vertical plane.
Minimality clearly implies stability, but the converse may not hold. 
Indeed, there are examples of stable intrinsic graphs in the first Heisenberg group $\bb H$ which are not area-minimizing;
see \cite[Theorem 7.1]{MR3984100} and \cite[Theorem 1.1]{zbMATH07541296}.

Theorem~\ref{thm669a693f} leaves open the discussion whether~\eqref{eq669a34fb} and the lower bound ``$q>4$'' are necessary to show that stability implies affinity of $f$.
Our main reason to have~\eqref{eq669a34fb} is to use the tools provided by~\cite{zbMATH07739119,zbMATH07620710}.
From those papers, we actually get uniqueness of the horizontal lines ruling the intrinsic graph of $f$.
More precisely, if $f$ satisfies the hypothesis of Theorem~\ref{thm669a693f} and $\Gamma_f$ is the intrinsic graph of $f$, we obtain from \cite{zbMATH07739119} that, for every point $p\in\Gamma_f$, there exists a unique horizontal curve in $\Gamma_f$ passing through $p$;
see Theorem~\ref{thm667be351} below and its reference.
Such uniqueness does not seem to be crucial in the proof of Theorem~\ref{thm669a693f};
see for instance \cite[Remark 3.7]{MR3753176} for an example with non-uniqueness of the horizontal curves.

We remark that having the intrinsic graph ruled by horizontal lines is not equivalent to being minimal.
On the one hand, there are minimal graphs that are not ruled by horizontal lines, as shown in~\cite{MR2333095}, see also~\cite{Ritore2009}.
On the other hand, there are graphs ruled by horizontal lines that are not minimal, as for instance the examples in~\cite{MR3984100}, or every ``graphical strip'' as in~\cite{MR2472175}.

It is an open question: Is it true that, if $f:\R^2\to\R$ is of class $C^1_{\bb W}$, that is, if both $f$ and $\grad^ff$ are continuous, and if the intrinsic graph of $f$ minimizes the sub-Riemannian area, then the intrinsic graph of $f$ is ruled by horizontal lines?

Finally, we remind the reader that for higher Heisenberg groups, very little is known, even for smooth intrinsic graphs; see \cite{SerraCassano2020} and references therein.

\subsection*{Plan of the paper}
In the preliminary section~\ref{sec06152116}, we introduce the main concepts and notations.
Section~\ref{sec68009c34} focuses on Bi-Sobolev Lagrangian homeomorphisms, which is a way to describe intrinsic graphs in the Heisenberg group in terms of their horizontal curves.
In Section~\ref{sec68009c4c} we show the consequences of the first variation.
Section~\ref{sec09022317} focuses on consequences of the second variation.
Finally, we will present a few examples in Section~\ref{sec68009c6b}.
We have added a short Appendix~\ref{sec68009c83} where we describe the relations between several conditions on the exponents of integrability that appear along the paper.

\subsection*{Acknowledgments}
SNG was partially supported 
by the Swiss National Science Foundation (grant 200021-204501 ‘{\it Regularity of sub-Riemannian geodesics and applications}’), 
by the European Research Council (ERC Starting Grant 713998 GeoMeG `{\it Geometry of Metric Groups}'), 
by the Academy of Finland 
(grant 288501 `{\it Geometry of subRiemannian groups}', 
grant 322898 `{\it Sub-Riemannian Geometry via Metric-geometry and Lie- group Theory}', 
grant 328846 `{\it Singular integrals, harmonic functions, and boundary regularity in Heisenberg groups}' 
grant 314172 `{\it Quantitative rectifiability in Euclidean and non-Euclidean spaces}'). 

FSC is member of Gruppo Nazionale per l’Analisi
Matematica, la Probabilità e le loro Applicazioni (GNAMPA), of the Istituto Nazionale di Alta Matematica (INdAM).
In particular, FSC was partially supported by
INdAM–GNAMPA 2025 Project ``Structure of sub-Riemannian hypersurfaces in Heisenberg groups'';
INdAM–GNAMPA 2023 Project ``Equazioni diﬀerenziali alle derivate parziali di tipo misto o dipendenti da campi di vettori'';
INdAM–GNAMPA 2022 Project ``Analisi geometrica in strutture subriemanniane''.
FSC was also partially supported by the European Union under {\it NextGenerationEU}. PRIN 2022, ``Regularity problems in sub-Riemannian structures'', Prot. n. 2022F4F2LH;
and by MIUR-PRIN 2017 Project ``Gradient flows, Optimal Transport and Metric Measure''.

\section{Preliminaries and notation}\label{sec06152116}
The Heisenberg group $\HH$ is represented in this paper as $\R^3$ endowed with the group operation
\[
(x,y,z)\cdot (x',y',z') = \left(x+x',y+y',z+z'+\frac12(xy'-x'y) \right) .
\]
In these coordinates, an orthonormal frame of the horizontal distribution is
\[
X = \de_x - \frac{y}{2} \de_z,
\qquad
Y = \de_y + \frac{x}{2} \de_z .
\]
The \emph{sub-Riemannian perimeter} of a measurable set $G\subset\HH$ in an open set $\Omega\subset\HH$ is
\[
P_{sR}(G;\Omega) = \sup\left\{ \int_G (X\psi_1+Y\psi_2) \did\cal L^3 : \psi_1,\psi_2\in C^\infty_c(\Omega),\ \psi_1^2+\psi_2^2\le 1 \right\} ,
\]
where $X\psi_1+Y\psi_2$ is the divergence of the vector field $\psi_1X+\psi_2Y$.
A set $G$ is a \emph{perimeter minimizer} in $\Omega\subset\HH$ if for every $F\subset\HH$ measurable with $(G\setminus F)\cup (F\setminus G)\Subset \Omega$, it holds $P_{sR}(G;\Omega) \le P_{sR}(F;\Omega)$.
A set $G$ is a \emph{local perimeter minimizer} in $\Omega\subset\HH$ if every $p\in\Omega$ has a neighborhood $\Omega'\subset\Omega$ such that $G$ is a perimeter minimizer in $\Omega'$.

Given a function $f:\omega\to\R$, $\omega\subset\R^2$, its intrinsic graph $\Gamma_f\subset\HH$ is the set of points
\[
(0,\eta,\tau)(f(\eta,\tau),0,0) = \left( f(\eta,\tau) , \eta , \tau-\frac12 \eta f(\eta,\tau) \right),
\]
for $(\eta,\tau)\in\omega$.
The \emph{intrinsic gradient} of $f$ is $\nabla^ff$, where $\nabla^f=\de_\eta+f\de_\tau$ is a vector field on $\omega$.
The function $\nabla^ff:\omega\to\R$ is well defined when $f\in W^{1,1}_\loc(\omega)\cap C^0(\omega)$.

\begin{remark}\label{rem669a6cca}
	If $f$ is of the form $f(\eta,\tau) = a\eta+b$
	for some $a,b\in\R$, then 
	\begin{align*}
	\Gamma_f 
	&= \left\{ \left( a\eta+b , \eta , \tau - a\eta^2/2 - b \eta/2 \right) :\ \eta,\tau\in\R \right\} \\
	&= (b,0,0) \cdot \left\{ \left( a\eta , \eta , z \right) :\ \eta,z\in\R \right\} \\
	&= \{(x,y,z) \in\R^3 : x= ay+b \}
	\end{align*}
	is a vertical plane.
\end{remark}

The \emph{graph area functional} is defined, for every $E\subset \omega$ measurable, by
\[
\scr A_f(E) := \int_E \sqrt{1+(\nabla^ff)^2} \did\cal L^2 .
\]
Such area functional descends from the perimeter measure of the graph, that is, $\scr A_f(E) = P_{sR}(G_f\cap (E\cdot\R))$, where $G_f:=\{(0,y,t)\cdot(\xi,0,0):\xi\le f(y,t)\}$ is the subgraph of $f$, and $E\cdot\R=\{(0,y,t)\cdot(\xi,0,0):\xi\in\R,\ (y,t)\in E\}$;
see \cite[Theorem 3.4]{MR2455341}.
A function $f\in W^{1,1}_\loc(\omega)\cap C^0(\omega)$ is \emph{(locally) area minimizing} if $G_f$ is (locally) perimeter minimizing in $\omega\cdot\R$.

We say that $f\in W^{1,1}_\loc(\omega)$ is \emph{stationary} if for all $\phi\in C^\infty_c(\omega)$
\begin{equation}\label{eq669a073c}
I_f(\phi) := \left.\frac{\dd}{\dd\epsilon}\scr A_{f+\epsilon\phi}(\spt(\phi))\right|_{\epsilon=0} = 0 .
\end{equation}
We say that $f\in W^{1,1}_\loc(\omega)$ is \emph{stable} if it is stationary and for all $\phi\in C^\infty_c(\omega)$
\begin{equation}\label{eq669a0749}
II_f(\phi) := \left.\frac{\dd^2}{\dd\epsilon^2}\scr A_{f+\epsilon\phi}(\spt(\phi))\right|_{\epsilon=0} \ge 0 .
\end{equation}
The functions $I_f$ and $II_f$ are called \emph{first} and \emph{second variation} of $f$, respectively.
It is clear that, if $f$ is a local area minimizer, then it is stable.

By~\cite[Remark 3.9]{MR2455341}, if $f\in W^{1,1}_\loc(\omega)\cap C^0(\omega)$, for some $\omega\subset\R^2$ open, then
\begin{align*}
	I_f(\phi) &= - \int_{\omega} \frac{\grad^ff}{\sqrt{1+(\grad^ff)^2}} (\grad^f\phi+\de_tf\,\phi) \did\LL^2 , \\
	II_f(\phi) &=\int_{\omega} 
			\left[ \frac{(\grad^f\phi+\de_tf\,\phi)^2}{\left( 1+(\grad^ff)^2 \right)^{3/2}}
			+  \frac{\grad^ff}{\sqrt{1+(\grad^ff)^2}} \de_t(\phi^2)\right]
		\did\LL^2 ,
\end{align*}
for all $\phi\in C^\infty_c(\omega)$.
Notice that the formal adjoint of $\grad^f$ is $(\grad^f)^*\phi=-\grad^f\phi-\de_tf\,\phi$.
By means of the triangle and the Hölder inequalities, one can easily show the following lemma.
We will denote by $W^{1,p}_0(\omega)$ is the closure of $C^\infty_c(\omega)$ in $W^{1,p}(\omega)$, 
and by $W^{1,p}_c(\omega)$ the subspace of functions in $W^{1,p}(\omega)$ whose essential support is compactly contained in $\omega$.

\begin{lemma}\label{lem669a07cf}
	Let $q>1$.
	A function $f\in W^{1,q}_\loc(\omega)$ with $\omega\subset\R^2$ open, is stationary 
	if and only if $I_f(\phi) = 0$ for all $\phi\in W^{1,q'}_c(\omega)$, 
	where $q' = \frac{q}{q-1}$ is the Hölder conjugate exponent.
\end{lemma}
\begin{proof}
	Notice that, using the fact that $\left|\frac{s}{\sqrt{1+s^2}}\right|\le 1$ for all $s\in\R$,
	and applying the Hölder inequality, there is $C>0$ such that, for all $\omega'\Subset\omega$ and all $\phi\in W^{1,q'}(\omega)$,
	\begin{equation}\label{eq688a4156}
	\left| \int_{\omega'} \frac{\grad^f f}{\sqrt{1+(\grad^f f)^2}} (\grad^f \phi + (\de_t f) \phi ) \did \mathscr{L}^2 \right|
	\le C (1+\|f\|_{W^{1,q}(\omega')}) \cdot \|\phi\|_{W^{1,q'}(\omega')} ,
	\end{equation}
	where $1/q + 1/q' = 1$.
	
	If $\phi\in W^{1,q'}_c(\omega)$, then there is $\omega'\Subset\omega$ open with $\spt_e(\phi)\Subset\omega'$, where $\spt_e$ denotes the essential support.
	Therefore, using standard mollifiers, one can show that there exists a sequence $\{\phi_k\}_{k\in\N}\subset C^\infty_c(\omega')$ converging to $\phi$ in $W^{1,q'}_c(\omega')$.
	From~\eqref{eq688a4156}, we obtain that
	$\lim_{k\to\infty} I_f(\phi_k-\phi) = 0$, and so $I_f(\phi)=0$.
\end{proof}
	
\begin{lemma}\label{lem66995925}
	Suppose that $f\in W^{1,q}_\loc(\omega)$ for $q>2$ and that $II_f(\phi)\ge0$ for all  $\phi\in C^\infty_c(\omega)$.
	Then $II_f(\phi)\ge0$ for all $\phi\in W^{1,\beta}_c(\omega)$, with $\beta = \frac{2q}{q-2} \in (2,\infty)$.
\end{lemma}
\begin{proof}
	First of all, we remark that, since $q/2>1$ then $\beta = 2(q/2)' = \frac{2q}{q-2} \in (2,\infty)$.
	Moreover, if $A,B$ are two non-negative measurable functions on $\omega$, then
	\begin{equation}\label{eq67f2415e}
	\begin{aligned}
		\|AB\|_{L^2(\omega)} 
		&= \left(\int_\omega A^2 B^2 \did x\right)^{1/2} \\
		&\le \left[\left(\int_\omega A^{2\frac{q}{2}} \did x \right)^{\frac{2}{q}}  \left(\int_\omega B^{2\frac{q}{q-2}} \did x\right)^{\frac{q-2}{q}} \right]^{1/2}
		= \|A\|_{L^q(\omega)} \|B\|_{L^\beta(\omega)} .
	\end{aligned}
	\end{equation}
	
	Notice that, for $\phi,\psi\in C^\infty_c(\omega)$,
	\begin{align*}
	 	\de_\tau\phi^2 - \de_\tau\psi^2
		&= 2(\phi\de_\tau\phi - \psi\de_\tau\psi) \\
		&= (\phi-\psi) \de_\tau (\phi+\psi) + (\phi+\psi) \de_\tau(\phi-\psi) .
	\end{align*}	
	
	Moreover, if $\phi,\psi\in C^\infty_c(\omega)$ are such that $\spt(\phi)\cup\spt(\psi)\subset K\Subset\omega$, then
	\begin{align*}
	&\left| II_f(\phi)-II_f(\psi) \right| \\
	&\le \int_\omega
		\left( \left| (\grad^f\phi+\de_\tau f \phi)^2 - (\grad^f\psi+\de_\tau f \psi)^2 \right|
			+ \left| \de_\tau (\phi^2) - \de_\tau (\psi^2) \right| \right)
		\did\LL^2 \\
	&\le \|(\grad^f\phi+\de_\tau f \phi)-(\grad^f\psi+\de_\tau f \psi)\|_{L^2}
		\cdot \|(\grad^f\phi+\de_\tau f \phi)+(\grad^f\psi+\de_\tau f \psi)\|_{L^2} \\
	&\qquad	+ \|\de_t(\phi-\psi) \|_{L^2} \|(\phi+\psi) \|_{L^2} + \|(\phi-\psi) \|_{L^2} \|\de_\tau (\phi+\psi) \|_{L^2} \\
	&\le ( \|\de_\eta\phi-\de_\eta\psi\|_{L^2} + \|f(\de_\tau\phi - \de_\tau\psi)\|_{L^2} + \|\de_\tau f (\phi-\psi)\|_{L^2} ) \cdot \\
	&\qquad\cdot ( \|\de_\eta\phi\|_{L^2} + \|f\de_\tau\phi\|_{L^2} + \| \de_\tau f \phi\|_{L^2} 
		+ \|\de_\eta\psi\|_{L^2} + \|f\de_\tau\psi\|_{L^2} + \| \de_\tau f \psi\|_{L^2}  ) \\
	&\qquad	+ \|\de_t(\phi-\psi) \|_{L^2} \|(\phi+\psi) \|_{L^2} + \|(\phi-\psi) \|_{L^2} \|\de_\tau (\phi+\psi) \|_{L^2} \\
	&\overset{\eqref{eq67f2415e}}\le ( |K|^{1/q} \|\de_\eta\phi-\de_\eta\psi\|_{L^\beta} 
		+ \|f\|_{L^q} \|\de_\tau\phi - \de_\tau\psi\|_{L^\beta}
		+  \|\de_\tau f\|_{L^q} \| (\phi-\psi)\|_{L^\beta} ) \cdot \\
		&\qquad\cdot ( |K|^{1/q} \|\de_\eta\phi\|_{L^2} 
		+  \|f\|_{L^q} \|\de_\tau\phi\|_{L^\beta} 
		+ \|\de_\tau f\|_{L^q} \| \phi\|_{L^\beta}   \\
		&\qquad\qquad
		+ |K|^{1/q} \|\de_\eta\psi\|_{L^\beta} 
		+ \|f\|_{L^q} \|\de_\tau\psi\|_{L^\beta} 
		+ \|\de_\tau f\|_{L^q} \|  \psi\|_{L^\beta}  ) \\
	&\qquad	+ |K|^{2/q} \|\de_t(\phi-\psi) \|_{L^\beta} \|(\phi+\psi) \|_{L^\beta}
		+ |K|^{2/q} \|(\phi-\psi) \|_{L^\beta} \|\de_\tau (\phi+\psi) \|_{L^\beta} \\
	&\le C ( |K|^{2/q} + \|f\|_{W^{1,q}(K)} + \|\phi\|_{W^{1,\beta}(K)} + \|\psi\|_{W^{1,\beta}(K)} )^2 \cdot \|\phi-\psi\|_{W^{1,\beta}(K)} ,
	\end{align*}
	for a constant $C$ depending on $q$.
	Then, arguing as in the proof of Lemma~\ref{lem669a07cf}, the thesis follows.
\end{proof}

\section{Bi-Sobolev Lagrangian homeomorphism}\label{sec68009c34}

\subsection{Continuity of Sobolev composition}

\begin{definition}[Homeomorphism of finite distortion]
	Given a map $\Psi\in W^{1,1}_\loc(\tilde\omega;\omega)$ between open sets $\tilde\omega,\omega\subset\R^2$
	 and $q>0$, define
	\[
	K^\Psi_q := \frac{|D\Psi|^q}{J_\Psi} ,
	\]
	where $|D\Psi(x)|$ and $J_\Psi(x)$ are the operator norm and the determinant of the matrix $D\Psi(x)$, respectively.
	
	A homeomorphism $\Psi:\tilde\omega\to\omega$ between open subsets of $\R^2$ 
	is \emph{of finite distortion}
	if $\Psi\in W^{1,1}_\loc(\tilde\omega;\omega)$
	and if $K^\Psi_1(x)$ is well defined and finite for almost every $x\in\tilde\omega$.
\end{definition}

\begin{lemma}[{\cite[Lemma 2.7]{MR3184742}}]\label{lem66964d4b}
	Let $\Psi:\tilde\omega\to\omega$ be an homeomorphism between open subsets of $\R^2$, $q>0$ and $p>2$.
	Suppose that $\Psi\in W^{1,p}_\loc(\tilde\omega;\omega)$, $\Psi^{-1}\in W^{1,p}_\loc(\omega,\tilde\omega)$ and $J_\Psi(x)>0$ for a.e.~$x\in\tilde\omega$.
	Then $\Psi$ is a homeomorphism of finite distortion and 
	\begin{equation}
		K^\Psi_q \in L^r_\loc(\tilde\omega)
	\end{equation}
	with $r = \left( \frac{q}{p} + \frac{2}{p-2} \right)^{-1}$, which may be smaller than 1.
\end{lemma}

\begin{lemma}[{\cite[Theorem 5.13]{MR3184742}}]\label{lem66964d4c}
	Let $\tilde\omega,\omega\subset\R^2$ open sets, $1\le s\le q<\infty$,
	and let $\Psi\in W^{1,1}_\loc(\tilde\omega;\omega)$ be a homeomorphism of finite distortion satisfying
	\begin{equation}
		K^\Psi_q \in L^{\frac{s}{q-s}}(\tilde\omega) .
	\end{equation}
	Then the operator $T_\Psi:u\mapsto u\circ\Psi$ is continuous from $W^{1,q}_\loc(\omega)\cap C(\omega)$ into $W^{1,s}_\loc(\tilde\omega)$.
	
	Moreover, we have
	\begin{equation}\label{eq66964e2a}
	D(u\circ\Psi)(x) = Du(\Psi(x)) D\Psi(x)
	\qquad\text{ for a.e.~}x\in\omega,
	\end{equation}
	if we use the convention that $Du(\Psi(x)) \cdot 0 = 0$ even if $Du$ does not exists or it equals infinity at $\Psi(x)$.
\end{lemma}

\begin{proposition}[Continuity of composition]\label{prop66964d64}
	Fix $p>2$ and consider $q,s$ such that:
	\begin{enumerate}
	\item
	$1\le s<p$ and $q= \frac{p^2 s}{(p-s)(p-2)}$, or, equivalently,
	\item
	$\frac{p^2}{(p-1)(p-2)} \le q < \infty$ and $s=\frac{ pq(p-2)}{p^2+q(p-2)}$ .
	\end{enumerate}
	
	Let $\Psi:\tilde\omega\to\omega$ be a homeomorphism between open subsets of $\R^2$.
	Suppose that $\Psi\in W^{1,p}_\loc(\tilde\omega;\omega)$, $\Psi^{-1}\in W^{1,p}_\loc(\omega,\tilde\omega)$
	and $J_\Psi(x)>0$ for a.e.~$x\in\tilde\omega$.
	Then the operator $T_\Psi$ defined in Lemma~\ref{lem66964d4c} is continuous from $W^{1,q}_\loc(\omega)\cap C(\omega)$ into $W^{1,s}_\loc(\tilde\omega)\cap C(\tilde\omega)$, 
	and~\eqref{eq66964e2a} holds.
\end{proposition}
\begin{proof}
	 The choices of $q$ and $s$ imply $q>s$,
	 and that the quantity $r = \left( \frac{q}{p} + \frac{2}{p-2} \right)^{-1}$
	 from Lemma~\ref{lem66964d4b}
	 equals the exponent $\frac{s}{q-s}$ from Lemma~\ref{lem66964d4c}.
	 Therefore, this proposition is just application in sequence of Lemma~\ref{lem66964d4b} and Lemma~\ref{lem66964d4c}.
\end{proof}

\subsection{Sobolev change of variables}

\begin{definition}[Lusin (N) condition]
	A function $\Psi:\tilde\omega\to\R^2$ satisfies the \emph{Lusin (N) condition} if 
	it maps null sets to null sets, i.e., if
	\[
	\forall E\subset\tilde\omega:
	\ |E| = 0 \THEN |\Psi(E)| = 0 .
	\]
\end{definition}

\begin{proposition}[Change of variables {\cite[Theorem A.35]{MR3184742}}]\label{prop66977f11}
	Let $\Psi:\tilde\omega\to\omega$ be a homeomorphism between open subsets of $\R^2$.
	Suppose that $\Psi\in W^{1,1}_\loc(\tilde\omega;\omega)$ and the $\Psi$ satisfies that Lusin (N) condition.
	For every non-negative or summable Borel function $u:\omega\to\R$
	\[
	\int_{\tilde\omega} u\circ\Psi(x) \cdot J_\Psi(x) \did x
	= \int_\omega u(y) \did y .
	\]
\end{proposition}

\subsection{Bi-Sobolev Lagrangian homeomorphism}
In Definition~\ref{def67f25313} below, we use the following notation.
For $\omega\subset\R^2$
and $r\in\R$, we denote by $\omega_{1,r}$ and $\omega_{2,r}$ the two subsets of $\R$ defined by
\[
	\begin{aligned}
	\omega_{1,r} &:= \{ \tau \in \R : (r,\tau)\in \omega \}, \\
	\omega_{2,r} &:= \{ \eta \in \R : (\eta,r)\in \omega \}.
	\end{aligned}
\]

\begin{definition}[Bi-Sobolev Lagrangian homeomorphism]\label{def67f25313}
	Let $\tilde\omega, \omega \subset \R^2$ be open sets, and $f:\omega\to\R$.
	We say that $\Psi: \tilde\omega \to \omega$ is a \emph{$f$-Lagrangian homeomorphism} if:
		\begin{enumerate}[label=(\roman*)]
		\item 
			$\Psi$ is a homeomorphism;
		\item 
			$\Psi$ has the form  $\Psi (t, \zeta)= (t, \chi (t, \zeta)) $ with $\chi$ continuous and $\chi(t,\cdot):\omega_{1,t}\to\R$ non-decreasing for every $t$;
		\item 
			For every $\zeta \in \R$, the map $\Psi(\cdot, \zeta):\omega_{2,\zeta}\to\R^2$ is locally absolutely continuous and satisfies
			$\de_t \Psi(t,\zeta) = \grad^f (\Psi (t,\zeta))$, or equivalently,
		\[
		\de_t \chi(t,\zeta) = f(t, \chi (t,\zeta)) ,
		\]
		for a.e.~$t$.
		\end{enumerate}
	We say that a $f$-Lagrangian homeomorphism $\Psi$ is \emph{$p$-bi-Sobolev} for $ p\ge 1$ if both $\Psi \in W_\loc^{1, p} (\tilde\omega)$ and $\Psi^{-1} \in W_\loc^{1, p} (\omega)$.
\end{definition}

Notice that if $\Psi$ is a $p$-bi-Sobolev Lagrangian homeomorphism with $p>2$, then $\Psi$ is differentiable almost everywhere with
\[
D\Psi(t, \zeta) = 
\begin{pmatrix}1&0 \\ \de_t\chi(t, \zeta) & \de_\zeta\chi(t, \zeta)\end{pmatrix} ,
\]
where $J_\Psi =  \de_\zeta\chi >0$ almost everywhere,
see~\cite[\S5.8.3]{Evans2010}.
In particular, if $u:\omega\to\R$, then
\begin{equation}\label{eq66967f49}
\begin{aligned}
	\de_{t} (u \circ \Psi) &= (\grad^f u) \circ \Psi , \\
	\de_\zeta (u \circ \Psi) &= (\de_\tau u \circ \Psi) \cdot \de_\zeta \chi .
\end{aligned}
\end{equation} 

We claim that 
\begin{equation}\label{eq66979337}
	\frac{1}{\de_\zeta\chi} \in L^\alpha_\loc(\tilde\omega) ,
	\qquad\text{with } \alpha = p-1 .
\end{equation}
Indeed, $1/\de_\zeta\chi = J_{\Psi^{-1}}\circ\Psi$, where $J_{\Psi^{-1}}\in L^{p}_\loc(\omega)$.
Notice that, a priori, we only have $J_{\Psi^{-1}}\in L^{p/2}_\loc(\omega)$,
but, by the specific form of $\Psi$, $D\Psi^{-1}$ is of the form 
$\begin{pmatrix}1&0 \\ A & B\end{pmatrix}$ for some functions $A,B\in L^p_\loc(\omega)$.
So, $J_{\Psi^{-1}} = B\in L^{p}_\loc(\omega)$.
Going back to the proof of~\eqref{eq66979337}, for every $K\Subset\tilde\omega$,
\begin{align*}
	\int_K \left|\frac{1}{\de_\zeta\chi}\right|^\alpha \did t \did\zeta
	&= \int_K \left|J_{\Psi^{-1}}\circ\Psi\right|^\alpha \frac{J_\Psi}{J_\Psi} \did t \did\zeta 
	= \int_{\Psi(K)} \left|J_{\Psi^{-1}}\right|^\alpha \frac{1}{J_\Psi\circ\Psi^{-1}} \did\eta \did\tau \\
	&= \int_{\Psi(K)} \left|J_{\Psi^{-1}}\right|^\alpha J_{\Psi^{-1}} \did\eta \did\tau 
	= \int_{\Psi(K)} \left|J_{\Psi^{-1}}\right|^{p} \did\eta \did\tau 
	<\infty .
\end{align*}

From the identity $\de_t \chi = f\circ\Psi$ and Proposition~\ref{prop66964d64}, we obtain that, 
if $f\in W^{1,q}_\loc(\omega)$ and if $\Psi$ is a $p$-bi-Sobolev $f$-Lagrangian homeomorphism with $p>2$,
then 
\begin{equation}\label{eq6697d093}
	\de_t \chi\in W^{1,s}_\loc(\tilde\omega)
	\qquad\text{for } s=\frac{ pq(p-2)}{p^2+q(p-2)} .
\end{equation}

Moreover, we can apply the identities~\eqref{eq66967f49} to deduce
\begin{equation}\label{eq66967fde}
\begin{aligned}
	\de_t^2\chi
	&= (\grad^f f) \circ \Psi , \\
	\de_\zeta \de_t \chi &= (\de_\tau f \circ \Psi) \cdot \de_\zeta \chi.
\end{aligned}
\end{equation} 

\begin{proposition}\label{prop6697d3da}
Let $\omega \subset\R^2$ be open, 
$f\in  C^0(\omega)$,
and $\Psi \in W^{1,1}_\loc (\tilde\omega ; \omega)$ be a Sobolev $f$-Lagrangian homeomorphism. 
Then $\Psi$ satisfies the Lusin~$(N)$ condition.
\end{proposition}
\begin{proof}
Let $E\subset\tilde\omega$ have $\mathscr{L}^2$-measure zero. 
For $t\in\R$, define $E_t=\{\zeta\in\R: (t,\zeta)\in E\}$.
Then $\LL^1(E_t)=0$ for a.e.~$t$.

Write $\Psi$ as $\Psi(t, \zeta) = (t, \chi(t, \zeta))$, with $\chi\in W^{1,1}_\loc({\tilde{\omega}})$. 
By the classical result \cite[Theorem 4.21\textit{(i)}]{Evans2015}, 
for almost every $t$ the map $\zeta\mapsto\chi(t,\zeta)$ is locally absolutely continuous,
and thus it satisfies the Lusin $(N)$ condition.
It follows that $\LL^1(\chi(t,E_t))=0$ for a.e.~$t$.
By Fubini-Tonelli Theorem, we conclude that $\LL^2(\Psi(E))=0$.
\end{proof}

\begin{theorem}\label{thm667be351}
	Let $I,J\subset\R$ be open intervals with $I$ of length $\ell>0$.
	Suppose that $f\in W^{1,1}_\loc(I\times J)\cap C^0(\bar I\times \bar J)$ and that 
	\begin{equation}
		\int_{I\times J} \exp\left( \frac{\ell p^2}{p-1} | \de_\tau f(\eta,\tau) | \right) \did \eta\did \tau < \infty .
	\end{equation}
	Then, for each point $(\eta_0,\tau_0)\in I\times J$, there exists a $p$-bi-Sobolev $f$-Lagrangian homeomorphism $\Psi:\tilde\omega\to \omega$ satisfying
	\begin{enumerate}
	\item
	$\tilde \omega = (\eta_0-\epsilon,\eta_0+\epsilon) \times (\tau_0 - \epsilon,\tau_0+\epsilon)$
	\item
	$\Psi(\eta,\tau) = (\eta,\chi(\eta,\tau))$ for all $(\eta,\tau)\in\tilde\omega$
	\item
	$\chi(\eta_0,\tau) = \tau$ for all $\tau\in(\tau_0 - \epsilon,\tau_0+\epsilon)$.
	\end{enumerate}
\end{theorem}
\begin{proof}
	Let $\epsilon_1>0$ be such that
	\[
	\omega^* := (\eta_0-\epsilon_1,\eta_0+\epsilon_1) \times (\tau_0-\epsilon_1,\tau_0+\epsilon_1) \Subset I\times J ,
	\]
	and let $\phi\in C^1_c(I\times J)$ be a cut-off function such that $0\le\phi\le 1$ in $I\times J$ and $\phi\equiv1$ in $\omega^*$.
	Consider $\hat f := \phi f$ on $I\times J$.
	Notice that $\hat f\in W^{1,1}(I\times J) \cap C^0_c(I\times J)$ and
	\begin{equation}\label{eq688a4863}
		\exp(|\de_\tau\hat f|)
		\le \exp\left( \| \de_\tau\phi\|_{L^\infty(\R^2)} \|f\|_{L^\infty(I\times J)} \right) \cdot \exp(|\de_\tau f|) 
	\end{equation}
	in $I\times J$.
	In particular, from~\eqref{eq688a4863} and~\eqref{eq66967fde}, it follows that
	\begin{equation}
		\exp(\beta|\de_\tau\hat f|) \in L^1(I\times J)
	\end{equation}
	with $\beta := \frac{\ell p^2}{p-1}$.
	We are now going to apply \cite[Theorem E]{zbMATH07739119}.
	Note that in \cite[Theorem E]{zbMATH07739119}, there is a distinction between time ``$t$'' and space ``$x$'' variables.
	In our setting,  ``$\tau$'' must be interpreted as the space variable and ``$\eta$'' as the time variable. 
	Under this convention, we can apply \cite[Theorem E]{zbMATH07739119} with $\Omega = J$ and $b=\hat f$, 
	and we get the existence of a unique flow $X:I\times I\times J\to J$ associated to $\hat f$ satisfying, for all $\eta,\bar\eta\in\bar I$, and all $\tau\in J$,
	\[
	\begin{cases}
		\frac{\dd X}{\dd \eta} (\eta,\bar\eta,\tau) = \hat f(\eta, X(\eta,\bar\eta,\tau)) \\
		X(\bar\eta,\bar\eta,\tau) = \tau .
	\end{cases} 
	\]
	and with $X(\eta,\bar\eta,\cdot)\in W^{1,p}(J)$ and
	\[
	\int_J |\de_\tau X(\eta,\bar\eta,\tau)|^p \did \tau
	\le \frac1\ell \int_J\int_I \exp(\beta|\de_\tau\hat f|) \did \eta\did\tau .
	\]
	Let's now define $\chi(\eta,\tau) := X(\eta,\eta_0,\tau)$ and
	$\Psi(\eta,\tau) := (\eta,\chi(\eta,\tau))$ if $(\eta,\tau)\in I\times J$.
	Then, it is easy to see that $\Psi:I\times J\to I\times J$ is a $p$-bi-Sobolev Lagrangian homeomorphism of $\hat f$ between $I\times J$ and itself.
	Notice that $\Psi(\eta_0,\tau_0) = (\eta_0,\tau_0)$.
	Therefore, there is $\epsilon>0$ so small that 
	\[
	\tilde \omega 
	:= (\eta_0-\epsilon,\eta_0+\epsilon) \times (\tau_0 - \epsilon,\tau_0+\epsilon)
	\subset \Psi^{-1}(\omega^*) .
	\]
	Set $\omega := \Psi(\tilde\omega)$.
	Then $\Psi:\tilde\omega\to\omega$ turns out to be the desired $p$-bi-Sobolev Lagrangian homeomorphism of $f$.
\end{proof}

\section{Consequence of the first variation}\label{sec68009c4c}
\begin{theorem}[Stationarity in Lagrangian coordinates]\label{thm6698ea1a}
	Let $\omega,\tilde\omega\subset \R^2$ be open sets, 
	$f\in W_\loc^{q}(\omega)\cap C^0(\omega)$ and 
	let $\Psi:\tilde\omega \to \omega$, $\Psi(t,\zeta) = (t, \chi(t,\zeta))$, be
	 a $p$-bi-Sobolev $f$-Lagrangian homeomorphism.
	Assume that $p>2$ and $\frac{p^2}{(p-1)(p-2)} \le q < \infty$.
	Set $s=\frac{ pq(p-2)}{p^2+q(p-2)}$, $\alpha = p-1$ and $q'=\frac{q}{q-1}$ and assume that
	\begin{equation}\label{eq6697d2f4}
	s > q'
	\qquad\text{and}\qquad
	\frac{1}{p} + \frac{1}{s} + \frac{2}{\alpha} \le 1 .
	\end{equation}
	Notice that these conditions correspond to~\eqref{eq67f239e4} and~\eqref{eq67f2503b},
	and thus they are not contradictory by Lemma~\ref{lem67f24a18}.
	
	If $f$ is stationary as in~\eqref{eq669a073c},
	then
 	\begin{equation}\label{eq66968599}
 	\int_{\tilde{\omega}} 
	\frac{\de^2_t\chi }{ \sqrt{ 1+ (\de^2_t\chi)^2 } } \de_t\theta \did\mathscr L^2 = 0
 	\end{equation}
	for all $\theta \in C^\infty_c(\tilde\omega)$.
\end{theorem}
\begin{proof}
	First, we show that
	\begin{equation}\label{eq66977e2d}
			\int_{\tilde\omega} 
			\frac{\de_t^2 \chi}{\sqrt{ 1 + (\de_t^2 \chi)^2 }}
			(\de_\zeta \chi \de_t u + u \de_\zeta \tilde{f} ) \did t \did\zeta=0 			
	\end{equation}
	for all $u\in C^\infty_c(\tilde\omega)$, where $\tilde f := f\circ\Psi = \de_t\chi$.
	
	Let $\tilde\phi\in C^\infty_c(\tilde\omega)$ and set $\phi := \tilde\phi\circ\Psi^{-1}$.
	By Proposition~\ref{prop66964d64}, $T_{\Psi^{-1}}:W^{1,q}_\loc(\tilde\omega)\cap C^0(\tilde\omega) \to W^{1,s}_\loc(\omega)\cap C^0(\omega)$ is continuous.
	Since $s>q'$ by~\eqref{eq6697d2f4},
	 then $W^{1,s}_\loc(\omega)\cap C^0(\omega) \subset W^{1,q'}_\loc(\omega)\cap C^0(\omega)$.
	Therefore, $\phi=T_{\Psi^{-1}}\tilde\phi \in W^{1,q'}_c(\omega)$.
	Since $f$ is stationary and by Lemma~\ref{lem669a07cf}, we have $I_f(\phi)=0$.
	Finally, we apply
	Propositions~\ref{prop6697d3da} and~\ref{prop66977f11},
	and the formulas~\eqref{eq66967f49} and \eqref{eq66967fde},
	to eventually obtain~\eqref{eq66977e2d}.
	
	Next, fix $\theta\in C^\infty_c(\tilde\omega)$ with $\spt(\theta)\subset K\Subset \tilde\omega$.

	We will perform a smooth mollification of $\chi$.
	For $\epsilon>0$, let $K_\epsilon := \{x\in\R^2 : {\rm dist}(x,K)<\epsilon\}$.
	Then $(K_\epsilon)_{\epsilon>0}$ is a family of bounded open sets containing $K$ and there exists $\epsilon_0>0$ such that $K_{2\epsilon}\Subset\tilde\omega$ for all $\epsilon\in(0,\epsilon_0)$.
	Let $\rho_\epsilon$ be a mollifier on $\R^2$ such that $\chi_\epsilon := \chi*\rho_\epsilon \in C^\infty(K_\epsilon)$ for $\epsilon\in (0,\epsilon_0)$.
	One can easily check the following properties.
	
	First, for $\epsilon$ small enough, $K_\epsilon$ is contained in $\tilde \omega$ and $\de_\zeta\chi_\epsilon >0$ on $K$.
	Second, 
	\begin{equation}\label{eq6697ce1d}
		0<\frac{1}{\de_\zeta \chi_\epsilon} \leq \left(\frac{1}{\de_\zeta \chi}\right)_\epsilon
		:= \left(\frac{1}{\de_\zeta \chi}\right) * \rho_\epsilon . 
	\end{equation}
	Indeed, the Hölder inequality implies
	\begin{align*}
	1 = \left\| \rho_\epsilon\right\| _{L^1(B(0,\epsilon))} 
	&\le \left\| \sqrt{\de_\zeta \chi(x-\cdot) \rho_\epsilon(\cdot)} \right\|_{L^2(B(0,\epsilon))}
		\left\| \sqrt{\frac{\rho_\epsilon(\cdot)}{\de_\zeta \chi(x-\cdot)}} \right\|_{L^2(B(0,\epsilon))}\\
	&=\sqrt{\de_\zeta \chi_\epsilon (x)} \cdot \sqrt{\left(\frac{1}{\de_\zeta \chi}\right)_\epsilon(x)} .
	\end{align*}
	Third, by~\eqref{eq6697ce1d} and by the fact that, by~\eqref{eq66979337}, 
	$\left(\frac{1}{\de_\zeta \chi}\right)_\epsilon \to \frac{1}{\de_\zeta \chi}$ in $L^\alpha(K)$,
	we have
	\begin{equation}\label{eq6697d0f1}
		\limsup_{\epsilon\to 0^+} \left\| \frac{1}{\de_\zeta \chi_\epsilon} \right\|_{L^{\alpha}(K)} 
		\le  \left\| \frac{1}{\de_\zeta \chi}\right\|_{L^{\alpha}(K)}.
	\end{equation}
	
	Next, we consider condition~\eqref{eq66977e2d} with the test function $u_\epsilon = \frac{\theta}{\de_\zeta\chi_\epsilon}$
	and we claim that 
	\begin{equation}\label{eq669784e5}
		\lim_{\epsilon\to0}\int_{\tilde\omega} 
		\frac{\de_t^2 \chi}{\sqrt{ 1 + (\de_t^2 \chi)^2 }}
		(\de_\zeta \chi \de_t u_\epsilon + u_\epsilon \de_\zeta \tilde{f} ) \did t \did\zeta
		= \int_{\tilde\omega} 
		\frac{\de_t^2 \chi}{\sqrt{ 1 + (\de_t^2 \chi)^2 }}
		\de_t\theta \did t \did\zeta .
	\end{equation}
	To show~\eqref{eq669784e5}, notice that $\left| \frac{\de_t^2 \chi}{\sqrt{ 1 + (\de_t^2 \chi)^2 }} \right| \le 1$ and that 
	\begin{multline*}
	(\de_\zeta \chi \de_t u_\epsilon + u_\epsilon \de_\zeta \tilde{f} ) - \de_t\theta \\
	= \de_{t} \theta \frac{\de_\zeta \chi - \de_\zeta \chi_\epsilon}{\de_\zeta \chi_\epsilon}
	+ \theta \frac{ \de_\zeta \de_{t} \chi (\de_\zeta \chi_\epsilon-\de_\zeta \chi)}{(\de_\zeta \chi_\epsilon)^2}
	+ \theta \frac{ \de_\zeta \chi \left(\de_\zeta \de_{t} \chi - \de_\zeta \de_{t} \chi_\epsilon\right)}{(\de_\zeta \chi_\epsilon)^2}.
	\end{multline*}
	Recall also that $\de_\zeta\chi\in L^p_\loc(\tilde\omega)$ because $\Psi\in W^{1,p}_\loc(\tilde\omega)$, 
	$\de_\zeta \de_{t} \chi\in L^s_\loc(\tilde\omega)$ by~\eqref{eq6697d093},
	and $\frac{1}{\de_\zeta\chi_\epsilon}\in L^\alpha_\loc(\tilde\omega)$ with $\alpha = p-1$
	by~\eqref{eq66979337}.
	Thanks to~\eqref{eq6697d2f4},
	the generalized Hölder inequality gives a constant $C=C(|K|)$ depending on the measure of $K$ such that 
	\begin{multline*}
	\left| \int_{\tilde\omega} 
		\frac{\de_t^2 \chi}{\sqrt{ 1 + (\de_t^2 \chi)^2 }}
		(\de_\zeta\chi \de_t u_\epsilon + u_\epsilon \de_\zeta \tilde{f} - \de_t\theta) \did t \did\zeta \right| \\
	\le \| \theta \|_{W^{1,\infty}(\tilde\omega)}
		\bigg( \left\|\frac{\de_\zeta\chi - \de_\zeta\chi_\epsilon}{\de_\zeta\chi_\epsilon} \right\|_{L^1(K)} \hfill\\
		+ \left\| \frac{ \de_\zeta \de_{t} \chi (\de_\zeta\chi_\epsilon-\de_\zeta\chi)}{(\de_\zeta\chi_\epsilon)^2} \right\|_{L^1(K)} \\
		\hfill+ \left\| \frac{ \de_\zeta\chi \left(\de_\zeta \de_{t} \chi - \de_\zeta \de_{t} \chi_\epsilon\right)}{(\de_\zeta\chi_\epsilon)^2} \right\|_{L^1(K)}
		 \bigg) \\
	\le C \| \theta \|_{W^{1,\infty}(\tilde\omega)}
		\bigg( 
		\left\|\frac{1}{\de_\zeta\chi_\epsilon} \right\|_{L^\alpha(K)}
		\| \de_\zeta\chi - \de_\zeta\chi_\epsilon \|_{L^p(K)}
		 \hfill\\
		+ \|\de_\zeta \de_{t} \chi \|_{L^s(K)} 
			\left\| \frac{ 1 }{(\de_\zeta\chi_\epsilon)^2} \right\|_{L^{\alpha/2}(K)}
			\|\de_\zeta\chi_\epsilon-\de_\zeta\chi\|_{L^p(K)}
			 \\
		\hfill+ \|\de_\zeta\chi\|_{L^p(K)} 
			\left\| \frac{ 1 }{(\de_\zeta\chi_\epsilon)^2} \right\|_{L^{\alpha/2}(K)}
			\|\de_\zeta \de_{t} \chi - \de_\zeta \de_{t} \chi_\epsilon\|_{L^s(K)}
		 \bigg) 
	\end{multline*}	
	Using~\eqref{eq6697d0f1} and the convergence of the mollification as $\epsilon\to0$, we obtain the limit in~\eqref{eq669784e5}.
\end{proof}

\begin{remark}\label{rem669a408d}
	Let $\omega,\tilde\omega\subset \R^2$ be open sets, 
	$f\in W_\loc^{1,q}(\omega)\cap C^0(\omega)$ and 
	let $\Psi:\tilde\omega \to \omega$, $\Psi(t,\zeta) = (t, \chi(t,\zeta))$, be
	 a $p$-bi-Sobolev $f$-Lagrangian homeomorphism.
	Suppose that~\eqref{eq66968599} is fulfilled for all $\theta \in C^\infty_c(\tilde\omega)$.
	
	Then \eqref{eq66968599} implies that $\de_\zeta^2\chi(t,\zeta)$ is constant in $t$ for almost every $\zeta$.
	Thereofore, up to precomposing $\Psi$ with a translation in $\R^2$,
	we have $\tilde\omega = (\eta_0-\epsilon,\eta_0+\epsilon) \times (\tau_0-\epsilon,\tau_0+\epsilon)$
	for some $(\eta_0,\tau_0)\in\omega$ and $\epsilon>0$,
	and 
	\begin{equation}\label{eq669912be}
	\chi(t,\zeta) = \frac{a(\zeta)}{2} (t-\eta_0)^2 + b(\zeta) (t-\eta_0) + c(\zeta) ,
	\end{equation}
	for some functions $a,b,c$.
	
	Moreover, we can show that 
	\[
	a,b,c\in W^{1,p}_\loc((\tau_0-\epsilon,\tau_0+\epsilon))\cap C^0((\tau_0-\epsilon,\tau_0+\epsilon)) .
	\]
	Indeed, fix three distinct values $t_1,t_2,t_3$ 
	such that
	the identity~\eqref{eq669912be} holds for almost every $\zeta$
	and the functions $\zeta\mapsto\chi(t_j,\zeta)$ are $W^{1,p}_\loc$.
	Their existence is ensured by the fact that $\chi\in W^{1,p}_\loc$.
	Then it is clear that each one of the functions $a,b,c$ are linear combinations of the continuous Sobolev functions
	$\chi(t_1,\cdot), \chi(t_2,\cdot), \chi(t_3,\cdot)$, 
	as the following computation makes explicit:
	\begin{align*}
	\begin{pmatrix}
	a(\zeta) \\
	b(\zeta) \\
	c(\zeta) 
	\end{pmatrix}
	&=
	\begin{pmatrix}
		\frac{t_{1}^{2}}{2} & t_{1} & 1\\
		\frac{t_{2}^{2}}{2} & t_{2} & 1\\
		\frac{t_{3}^{2}}{2} & t_{3} & 1
	\end{pmatrix}^{-1}
	\begin{pmatrix}
	\chi(t_1,\zeta) \\
	\chi(t_2,\zeta) \\
	\chi(t_2,\zeta) 
	\end{pmatrix} \\
	&= 
	\begin{pmatrix}
		\frac{2}{t_{1}^{2} - t_{1} t_{2} - t_{1} t_{3} + t_{2} t_{3}} 
			& - \frac{2}{t_{1} t_{2} - t_{1} t_{3} - t_{2}^{2} + t_{2} t_{3}} 
			& \frac{2}{t_{1} t_{2} - t_{1} t_{3} - t_{2} t_{3} + t_{3}^{2}}\\
		\frac{- t_{2} - t_{3}}{t_{1}^{2} - t_{1} t_{2} - t_{1} t_{3} + t_{2} t_{3}} 
			& \frac{t_{1} + t_{3}}{t_{1} t_{2} - t_{1} t_{3} - t_{2}^{2} + t_{2} t_{3}} 
			& \frac{- t_{1} - t_{2}}{t_{1} t_{2} - t_{1} t_{3} - t_{2} t_{3} + t_{3}^{2}}\\
		\frac{t_{2} t_{3}}{t_{1}^{2} - t_{1} t_{2} - t_{1} t_{3} + t_{2} t_{3}} 
			& - \frac{t_{1} t_{3}}{t_{1} t_{2} - t_{1} t_{3} - t_{2}^{2} + t_{2} t_{3}} 
			& \frac{t_{1} t_{2}}{t_{1} t_{2} - t_{1} t_{3} - t_{2} t_{3} + t_{3}^{2}}
	\end{pmatrix}
	\begin{pmatrix}
	\chi(t_1,\zeta) \\
	\chi(t_2,\zeta) \\
	\chi(t_2,\zeta) 
	\end{pmatrix} .
	\end{align*}
\end{remark}

Under the hypothesis of Theorem~\ref{thm667be351}, we have a stronger result.

\begin{corollary}\label{cor669a6101}
	Let $f\in W^{1,q}_\loc(\omega)\cap C^0(\omega)$ with $q>2$ and
	 such that~\eqref{eq669a34fb} holds.
	Suppose that $f$ is stationary.
	Then 
	for every $p\ge 1$ the following holds:
	for every $(\eta_0,\tau_0)\in\omega$ there are $\epsilon>0$ and 
	$\chi:(\eta_0-\epsilon,\eta_0+\epsilon)\times(\tau_0-\epsilon,\tau_0+\epsilon)\to\R$
	such that the map $\Psi(t,\zeta) = (t,\chi(t,\zeta))$ is a $p$-bi-Sobolev $f$-Lagrangian homeomorphism,
	where $\chi$ is of the form
	\[
	\chi(t,\zeta) = \frac{a(\zeta)}{2} (t-\eta_0)^2 + b(\zeta) (t-\eta_0) + \zeta ,
	\]
	for some $a,b\in W^{1,p}_\loc((\tau_0-\epsilon,\tau_0+\epsilon))\cap C^0((\tau_0-\epsilon,\tau_0+\epsilon))$.
	Moreover, for a.e.~$\zeta$, $a(\zeta) = \grad^ff(\eta_0,\zeta)$  and $b(\zeta) = f(\eta_0,\zeta)$.
	
	As a consequence, we have that $f,\grad^ff\in C^0(\omega)\cap W^{1,p}_\loc(\omega)$,
	for every $p\ge1$.
	Moreover, if $\omega = \R^2$, then we can take $\epsilon = \infty$;
	in particular, the function $\Psi(t,\zeta) = (t,\chi(t,\zeta))$ with 
	\begin{equation}\label{eq68aa2429}
	\chi(t,\zeta) = \frac{\grad^ff(0,\zeta)}{2} t^2 + f(0,\zeta) t + \zeta 
	\end{equation}
	is a $p$-bi-Sobolev $f$-Lagrangian homeomorphism $\R^2\to\R^2$, for every $p\ge1$.
\end{corollary}
\begin{proof}
	Let $p\ge1$.
	Since the condition of being $p$-bi-Sobolev is local, we assume $p>\hat p$,
	where $\hat p$ is chosen as in Lemma~\ref{lem67f24a18}. 
	Take $\epsilon>0$ such that $\kappa = \frac{2\epsilon p^2}{p-1}$:
	then~\eqref{eq669a34fb} and Theorem~\ref{thm667be351} imply the existence of such a
	$p$-bi-Sobolev $f$-Lagrangian homeomorphism $\Psi$.
	The form of $\chi$ and the regularity of $a,b$ come from Remark~\ref{rem669a408d}.
	The fact that $c(\zeta) = \zeta$ is a consequence of Theorem~\ref{thm667be351}.
	The formulas $a(\zeta) = \grad^ff(\eta_0,\zeta)$  and $b(\zeta) = f(\eta_0,\zeta)$ follow from~\eqref{eq66967fde}.
	This shows the first part of the corollary.
	
	The continuity of $\grad^ff$ follows from \eqref{eq66967fde} and the fact that $(\eta_0,\tau_0)$ is arbitrary in $\omega$.
	The Sobolev regularity of $f$ and $\grad^ff$ is more tricky to show, but not difficult:
	First, notice that, in a neighborhood of $(\eta_0,\tau_0)$,
	we have $f = \de_t\chi\circ\Psi^{-1} = (at+b)\circ\Psi^{-1}$ and $\grad^ff = \de_t^2\chi\circ\Psi^{-1} = a\circ\Psi^{-1}$,
	where $a$ and $b$ belong to the class $W^{1,p}_\loc$.
	Second, using the first part of the corollary, we notice that,
	up to taking a smaller neighborhood of $(\eta_0,\tau_0)$, 
	for $\bar q$ arbitrarily large,
	we have that $\Psi$ is $\bar q$-bi-Sobolev homeomorphism
	and both $a$ and $b$ are in the class $W^{1,\bar q}_\loc$.
	Third, we apply Proposition~\ref{prop66964d64} with $\bar p$ large enough, $\bar s:=\bar p/2$ and $\bar q$ given by the formula $\bar q := \frac{\bar p^2 \bar s}{(\bar p-\bar s)(\bar p-2)}$:
	we only need $\bar p$ to be large enough to have $\bar q \ge1$.
	Note that, since $\bar q>\bar p$, $\Psi$ is a $\bar p$-bi-Sobolev homeomorphism too.
	So, from Proposition~\ref{prop66964d64} we obtain that $f$ and $\grad^ff$ belong to $W^{1,\bar s}_\loc(U)$, for some neighborhood $U$ of $(\eta_0,\tau_0)$.
	Since we can take $\bar s$ larger than $p$, so that we conclude $f,\grad^ff\in W^{1,\bar s}_\loc(U)$.
	Since $(\eta_0,\tau_0)$ is arbitrary, we conclude $f,\grad^ff\in W^{1,p}_\loc(\omega)$.
	
	Next,
	suppose $\omega=\R^2$ and fix $(\eta_0,\tau_0)\in\R^2$.
	Without loss of generality, we assume $(\eta_0,\tau_0)=(0,0)$,
	so that $\Psi(t,\zeta) = (t,\chi(t,\zeta))$ with $\chi$ as in~\eqref{eq68aa2429}.
	The fact that the map $\Psi$ is a homeomorphism $\R^2\to\R^2$ is proven in \cite[Lemmas 3.3 and 3.5]{MR3753176}.
	We need to show that $\Psi,\Psi^{-1}\in W^{1,p}_\loc(\R^2;\R^2)$, each $p\ge 1$.
	Since $a,b\in W^{1,p}_\loc(\R)$, then $\Psi\in W^{1,p}_\loc(\R^2;\R^2)$.
	For the inverse map, notice that, if $(t,\zeta),(\eta,\tau)\in\R^2$ are such that $(\eta,\tau) = \Psi(t,\zeta)$, then $\eta = t$,
	$\tau = \chi(t,\zeta) = a(\zeta) t^2/2 + b(\zeta) t + \zeta$,
	$f(\eta,\tau) = f(t,\chi(t,\zeta)) = \de_t\chi(t,\zeta) = a(\zeta) t + b(\zeta)$,
	and $\grad^ff(\eta,\tau) = \grad^ff(t,\chi(t,\zeta)) = \de_t^2\chi(t,\zeta) = a(\zeta)$.
	Therefore,
	\begin{align*}
		\tau 
		&= a(\zeta) t^2/2 + b(\zeta) t + \zeta \\
		&= \grad^ff(\eta,\tau) t^2/2 + (f(\eta,\tau)  - a(\zeta) t) t + \zeta \\
		&= - \grad^ff(\eta,\tau) t^2/2 + f(\eta,\tau) t + \zeta ,
	\end{align*}
	that is, 
	\[
	\zeta = \tau + \grad^ff(\eta,\tau) \eta^2/2 - f(\eta,\tau)\eta .
	\]
	We conclude that $\Psi^{-1}(\eta,\tau) = (\eta, \tau + \grad^ff(\eta,\tau) \eta^2/2 - f(\eta,\tau)\eta)$.
	Since $f,\grad^ff\in W^{1,p}_\loc(\R^2)$, we obtain $\Psi^{-1}\in W^{1,p}_\loc(\R^2;\R^2)$.
\end{proof}

\section{Consequences of the second variation}\label{sec09022317}

Robert Young proved in \cite[Theorem 1.3]{zbMATH07541296} that, if the intrinsic graph $\Gamma_f$ of $f$ is area minimizing and ruled by horizontal straight lines (i.e., isometric embeddings of $\R$), then $\Gamma_f$ is a vertical plane.
In this section, we study graphs $\Gamma_f$ that are ruled by horizontal straight lines,
and we want deduce that $\Gamma_f$ is a vertical plane from the stability of $f$, not the minimality;
see Proposition~\ref{prop669a6320} below.
Minimality clearly implies stability, but the converse may not hold.

\begin{lemma}\label{lem04061138}
	Let $f\in W^{1,q}_\loc(\R^2)$ be a function with $p$-bi-Sobolev $f$-Lagrangian homeomorphism 
	$\Psi:\R^2\to\R^2$, $\Psi(t,\zeta) = (t, \chi(t,\zeta))$, with
	\[
	\chi(t,\zeta) = \frac{a(\zeta)}{2} t^2 + b(\zeta) t + \zeta .
	\]
	Assume that (cfr.~\eqref{eq67f239dd})
	\begin{equation}\label{eq66993851}
	\frac{p^2}{(p-1)(p-2)} < q < \infty .
	\end{equation}
	Then:
	\begin{enumerate}
	\item
	For all $\zeta_1,\zeta_2\in\R$, either $a(\zeta_1)=a(\zeta_2)$ and $b(\zeta_1) = b(\zeta_2)$,
 	or $2 \big(a(\zeta_1)-a(\zeta_2)\big) (\zeta_1-\zeta_2) > \big(b(\zeta_1)-b(\zeta_2)\big)^2 $;
	\item
	For a.e.~$\zeta\in\R$ we have either $a'(\zeta)= b'(\zeta)=0$, or $2a'(\zeta)> b'(\zeta)^2$.
	\end{enumerate}
\end{lemma}
\begin{remark}
	Notice that, for $q>1$ fixed, the condition~\eqref{eq66993851} is satisfied for $p$ large enough.
	See also Lemma~\ref{lem67f24a18}.
\end{remark}
\begin{proof}[Proof of Lemma~\ref{lem04061138}]
	The first part of the lemma is contained in \cite[Lemma~3.3]{MR3753176}.
	Before proving the second part, notice that $2a'(\zeta) \ge b'(\zeta)^2$ follows directly from the inequality $2 \big(a(\zeta_1)-a(\zeta_2)\big) (\zeta_1-\zeta_2) \ge \big(b(\zeta_1)-b(\zeta_2)\big)^2 $, which holds for every $\zeta_1,\zeta_2\in\R$.
	
	In order to show the second part of the lemma, notice that
	\[
	\de_\tau f( t,\chi(t,\zeta) ) = \frac{ \de_t\de_\zeta \chi(t,\zeta) }{ \de_\zeta\chi(t,\zeta) } .
	\]
	Since $\de_\tau f$ is locally $q$-integrable and since $\Psi(t,\zeta) = (t,\chi(t,\zeta))$ is $p$-bi-Sobolev homeomorphism, then $(t,\zeta)\mapsto \de_\tau f( t,\chi(t,\zeta) )$ is locally $s$-integrable by Proposition~\ref{prop66964d64}, where $s>1$ because of the strict inequality in~\eqref{eq66993851}.
	In particular, for almost every $\zeta$, the function $t\mapsto \frac{ \de_t\de_\zeta \chi(t,\zeta) }{ \de_\zeta\chi(t,\zeta) }$ is locally $s$-integrable in $t$.
	
	It follows that, for a.e.~$\zeta\in\R$, the polynomial $\de_\zeta\chi(t,\zeta) = a'(\zeta)t^2/2 + b'(\zeta) t+ 1$ does not have roots.
	Indeed, either $t\mapsto\de_\zeta\chi(t,\zeta)$ is constant or not.
	If $t\mapsto\de_\zeta\chi(t,\zeta)$ is constant, then it is 1, hence it does not have roots.
	If $t\mapsto\de_\zeta\chi(t,\zeta)$ is not constant, by Lemma~\ref{lem688b709b}, whenever $\zeta$ is such that $t\mapsto \de_\zeta\chi(t,\zeta)$ has a root, then there is $k_\zeta\in\Z$ such that 
	\[
	\int_{k_\zeta}^{k_\zeta+1}  \left|\frac{ \de_t\de_\zeta \chi(t,\zeta) }{ \de_\zeta\chi(t,\zeta) }\right|^s \did t = \infty .
	\]
	Since, for every $j,k\in\Z$ we have
	\[
	\int_j^{j+1} \int_k^{k+1}  \left|\frac{ \de_t\de_\zeta \chi(t,\zeta) }{ \de_\zeta\chi(t,\zeta) }\right|^s \did t \did\zeta < \infty ,
	\]
	it follows that for every $j,k\in\Z$ the set $\{\zeta\in[j,j+1]:k_\zeta =k\}$ has zero measure.
\end{proof}

\begin{lemma}\label{lem688b709b}
	Let $P:\R\to\R$ be a non-constant polynomial of degree 1 or 2.
	For every $a,b\in\R$, $a<b$, if $P$ has a root in $[a,b]$, then, for every $s\ge1$,
	\begin{equation}\label{eq688b6f8f}
		\int_a^b \left|\frac{ P'(t) }{ P(t) }\right|^s \did t = \infty .
	\end{equation}
\end{lemma}
\begin{proof}
	If $P$ has degree 1, then $P(t) = A(t-B)$ for some $A,B\in\R$ with $A\neq0$, so $\frac{ P'(t) }{ P(t) } = \frac{1}{t-B}$.
	Therefore, we have~\eqref{eq688b6f8f} whenever $B\in[a,b]$.

	If $P$ has degree 2 and has a real root, then $P(t) = A(t-\alpha) (t-\beta)$ for some $A,\alpha,\beta\in\R$ with $A\neq0$.
	So, $\frac{ P'(t) }{ P(t) } = \frac{2t - \alpha-\beta}{(t-\alpha) (t-\beta)}$ and we obtain~\eqref{eq688b6f8f} as soon as $[a,b]$ contains $\alpha$ or $\beta$.
\end{proof}

\begin{lemma}\label{lem04061131}
	Let $f\in W^{1,q}_\loc(\R^2)$ be a function with $p$-bi-Sobolev $f$-Lagrangian homeomorphism 
	$\Psi:\R^2\to\R^2$, $\Psi(t,\zeta) = (t, \chi(t,\zeta))$ 
	as in Lemma~\ref{lem04061138}.
	Assume also that~\eqref{eq67f25035} holds, that is,
	\begin{equation}\label{eq669959c6}
	s := \frac{ pq(p-2)}{p^2+q(p-2)}  > \frac{2q}{q-2} =: \beta .
	\end{equation}
	
	If $f$ is stable, then, for all $\theta\in C^\infty_c(\R^2)$,
	\begin{equation}\label{eq06071139}
	\int_{\R^2} 
		\left[ (\de_t\theta)^2 \frac{a't^2/2 + b't + 1}{(1+a^2)^{3/2}} 
	- \theta^2 \frac{2a'-b'^2}{(a't^2/2 + b't + 1)(1+a^2)^{3/2}}  
	\right]
		\did t\did \zeta
	\ge 0 ,
	\end{equation}
	where $a$, $a'$ and $b'$ are functions of $\zeta$, while $\theta$ is a function of $(t,\zeta)$.
\end{lemma}

\begin{remark}
	By Lemma~\ref{lem67f24a18},~\eqref{eq669959c6} is satisfied for $p$ large enough
	as soon as $q>4$.
\end{remark}

\begin{proof}[Proof of Lemma~\ref{lem04061131}]
	Let $\theta\in C^\infty_c(\R^2)$ and set $\phi = \theta\circ\Psi^{-1}$.
	By Proposition~\ref{prop66964d64}, $T_{\Psi^{-1}}$ is continuous 
	$W^{1,q}_\loc(\R^2)\cap C^0(\R^2) \to W^{1,s}_\loc(\R^2)\cap C^0(\R^2)$.
	Since $s>\beta$ by~\eqref{eq669959c6},
	 then $W^{1,s}_\loc(\R^2)\cap C^0(\R^2) \subset W^{1,\beta}_\loc(\R^2)\cap C^0(\R^2)$.
	Therefore, $\phi=T_{\Psi^{-1}}\theta \in W^{1,\beta}_c(\R^2)$.
	By Lemma~\ref{lem66995925} and the hypothesis of $f$ being stable,
	$II_f(\phi) \ge 0$.
	
	With the help of Propositions~\ref{prop6697d3da} and~\ref{prop66977f11},
	and the formulas~\eqref{eq66967fde} and~\eqref{eq66967f49},
	one eventually obtains~\eqref{eq06071139}:
	\begin{multline*}
	\!\!\!\!II_f(\phi)\!
	=\!\!\! \int_{\R^2} \!\! \left(\frac{
		\left(\de_t\theta+\theta\frac{a't+b'}{a't^2/2 + b't + 1}\right)^2 }{(1+a^2)^{3/2}} + 
		\frac{
			\frac{\de_\zeta(\theta^2)}{a't^2/2 + b't + 1} a
			}{(1+a^2)^{1/2}}
		\right)\!\!
		(a't^2/2 + b't + 1)
		\did t \did\zeta \\
	= \int_{\R^2} \bigg[ (\de_t\theta)^2 \frac{a't^2/2 + b't + 1}{(1+a^2)^{3/2}}  
		+ \theta^2 \frac{(a't+b')^2}{(a't^2/2 + b't + 1)(1+a^2)^{3/2}}  +\hfill\\ \hfill
		+ \de_t(\theta^2) \frac{a't+b'}{(1+a^2)^{3/2}}
		+ \de_\zeta(\theta^2)\frac{a}{(1+a^2)^{1/2}} \bigg]\did t\did \zeta \\
	= \int_{\R^2} \bigg[ (\de_t\theta)^2 \frac{a't^2/2 + b't + 1}{(1+a^2)^{3/2}}  
		+ \theta^2 \frac{(a't+b')^2}{(a't^2/2 + b't + 1)(1+a^2)^{3/2}}  + \hfill\\\hfill
		- \theta^2 \frac{a'}{(1+a^2)^{3/2}}
		- \theta^2 \frac{a'}{(1+a^2)^{3/2}} \bigg] \did t\did \zeta \\
	= \int_{\R^2} \bigg[ (\de_t\theta)^2 \frac{a't^2/2 + b't + 1}{(1+a^2)^{3/2}}  
		+ \theta^2 \frac{b'^2-2a'}{(a't^2/2 + b't + 1)(1+a^2)^{3/2}}  
		\bigg]\did t\did \zeta . \hfill
	\end{multline*}
	See also \cite[Lemma 5.4]{MR3984100}. 
\end{proof}

The following lemma is proved in \cite[p.45]{MR2333095}.
\begin{lemma}\label{lem06070850}
	Let $A,B\in\R$ be such that $B^2\le2A$ and set $h(t):=At^2/2 + Bt + 1$.
	If
	\begin{equation*}
	\int_{\R} \phi'(t)^2 h(t) \did t 
	\ge (2A-B^2) \int_{\R} \phi(t)^2 \frac1{h(t)} \did t
	\qquad
	\forall \phi\in C^1_c(\R),
	\end{equation*}
	then $B^2=2A$.
\end{lemma}

\begin{proposition}\label{prop669a6320}
	Let $f\in W^{1,q}_\loc(\R^2)$ be a function with $p$-bi-Sobolev $f$-Lagrangian homeomorphism 
	$\Psi:\R^2\to\R^2$, $\Psi(t,\zeta) = (t, \chi(t,\zeta))$, with
	\[
	\chi(t,\zeta) = \frac{a(\zeta)}{2} t^2 + b(\zeta) t + \zeta .
	\]
	Assume both~\eqref{eq66993851} and~\eqref{eq669959c6}.
	If $f$ is stable, then $a'=b'=0$ a.e.~and thus the intrinsic graph of $f$ is a vertical plane.
\end{proposition}
\begin{proof}
	We can apply in sequence Lemma~\ref{lem04061138}, Lemma~\ref{lem04061131} and finally Lemma~\ref{lem06070850},
	arguing as in~\cite[Proof of Theorem 5.1 at p. 141]{MR3984100}.
\end{proof}

\begin{proof}[Proof of Theorem~\ref{thm669a693f}]
	Notice that, by Morrey Embedding Theorem, $f\in C^0(\R^2)$.
	Being $f$ is stationary, we apply Corollary~\ref{cor669a6101}:
	we obtain a $f$-Lagrangian homeomorphism $\Psi:\R^2\to\R^2$ of the form $\Psi(t,\zeta) = (t,\chi(t,\zeta))$
	with 
	\[
	\chi(t,\zeta) = \frac{a(\zeta)}{2} t^2 + b(\zeta) t + \zeta ,
	\]
	for some $a,b\in W^{1,p}_\loc(\R)\cap C^0(\R)$.
	Moreover, the map $\Psi$ is $p$-bi-Sobolev for every $p\ge1$.
	By Lemma~\ref{lem67f24a18}, the hypothesis $q>4$ ensures that we can choose $p$ so that both~\eqref{eq66993851} and~\eqref{eq669959c6} are verified,
	and thus we can apply Proposition~\ref{prop669a6320}.
	We obtain that $a'=b'=0$.
	We conclude that $f(\eta,\tau) = a\eta + b$ is affine.
\end{proof}

\section{Examples}\label{sec68009c6b}

\subsection{A Sobolev function that is not intrinsic Lipschitz}
We show that our Theorem~\ref{thm669a693f} is indeed a generalization of our previous result~\cite[Theorem 1.1]{MR3984100} for intrinsic Lipschitz functions.
In other words, we want to show that there are functions $f\in W^{1,p}_\loc(\R^2)$ for all $p>1$, satisfying~\eqref{eq669a34fb} that are not intrinsic Lipschitz.

If $f:\R^2\to\R$ is a function of the form $f(\eta,\tau) =g(\tau)$ for some absolutely continuous $g:\R\to\R$, then $f$ is \emph{(locally) intrinsic Lipschitz} in the Heisenberg group
if and only if $\grad^ff(\eta,\tau) = g(\tau)g'(\tau)$ is (locally) bounded.
At the same time, such $f$ belongs to $W^{1,p}_\loc(\R^2)$ if and only if $g'\in L^p_\loc(\R)$, for each $p>1$.

Consider for instance
\[
g(\tau) = \begin{cases}
1+\tau\log(\tau) & \text{ if }\tau>0 , \\
0 & \text{ otherwise} .
\end{cases}
\]
Then $g'\in L^p_\loc(\R)$ for each $p>1$, but, 
if $\tau>0$,
we have 
\[
	\grad^ff(\eta,\tau) 
	= g(\tau)g'(\tau) 
	= (1+\tau\log(\tau))(\log(\tau)+1), 
\]
which is not bounded on $[0,1]$.
So, this is an example of a function $f\in W^{1,p}_\loc(\R^2)$ for all $p>1$ that satisfied~\eqref{eq669a34fb}, but which is not intrinsic Lipschitz.

\subsection{A known example that does not satisfy our hypothesis}
In \cite[Section 2]{MR2455341}, it has been shown that the locally intrinsic Lipschitz graph in $\HH$ of the function 
\begin{equation}\label{eq6800c406}
	f(\eta,\tau) = 2 \operatorname{sgn}(\tau) \sqrt{|\tau|} 
\end{equation}
is an area minimizing surface.
Note that~\cite{MR2455341} uses a different coordinate system for the Heisenberg group: we have taken it into account when writing the function $f$ in~\eqref{eq6800c406}.\footnote{
For sake of completeness, we explain the change of coordinates.
Consider the two group operations $*$ and $\bullet$ on $\R^3$ defined by
\begin{align*}
	(a,b,c) * (x,y,z) &:= ( a+x , b+y , c+z+\frac12 (ay-bx) ) , \\
	(a,b,c) \bullet (x,y,z) &:= ( a+x , b+y , c+z-2 (ay-bx) ) .
\end{align*}
The group operation $*$ is the one used in the present paper, see Section~\ref{sec06152116}, while $\bullet$ is used in~\cite{MR2455341}.
The map $\Phi:\R^3\to\R^3$, $\Phi(x,y,z):=(x,y,-z/4)$, is an isomorphism $\Phi:(\R^3,\bullet) \to (\R^3, *)$.
For $\alpha\in\R$, consider the two subsets of $\R^3$
\begin{align*}
	\Sigma^*_\alpha &:= \{(0,\eta,\tau)*(\alpha \operatorname{sgn}(\tau) \sqrt{|\tau|} , 0,0) : \eta,\tau\in\R\} \\
	&= \{ ( \alpha \operatorname{sgn}(\tau) \sqrt{|\tau|} , \eta, \tau - \frac12 \alpha \operatorname{sgn}(\tau) \sqrt{|\tau|}\eta ) : \eta,\tau\in\R\} , \\
	\Sigma^\bullet_\alpha &:= \{(0,\eta,\tau)\bullet(\alpha \operatorname{sgn}(\tau) \sqrt{|\tau|} , 0,0) : \eta,\tau\in\R\} \\
	&= \{ ( \alpha \operatorname{sgn}(\tau) \sqrt{|\tau|} , \eta, \tau + 2 \alpha \operatorname{sgn}(\tau) \sqrt{|\tau|}\eta ) : \eta,\tau\in\R\} .
\end{align*}
The example given by~\cite{MR2455341} is $\Sigma^\bullet_{-1}$, that is, the intrinsic graph in $(\R^3,\bullet)$ of the function $(\eta,\tau) \mapsto -\operatorname{sgn}(\tau) \sqrt{|\tau|}$.
One can check that $\Phi(\Sigma^\bullet_{-1}) = \Sigma^*_{2}$,
which is the intrinsic graph in $(\R^3,*)$ of the function given in~\eqref{eq6800c406}.
}

In particular, $f$ is stable, but not affine.
Our Theorem~\ref{thm669a693f} does not apply to this function $f$ because both $|\de_\tau f|^4$ is not locally integrable, and $\exp(\kappa|\de_\tau f|)$ is not locally integrable for every $\kappa>0$.
Note that, by~\cite[Theorem 1.3]{zbMATH07541296}, $\Gamma_f$ cannot be ruled by horizontal lines.

\appendix
\section{Consistency of conditions on exponents}\label{sec68009c83}

Along the paper, we need a number of conditions on the integrability exponents.
The following Lemma~\ref{lem67f24a18} ensures not only that those conditions are not contradictory, 
but also that, starting with the initial exponent $q>2$, 
we can pick exponents $p$, $r$, $s$, etc... along the way that satisfy all requirements.

Notice that $\frac{2q}{q-2}$ is not quite the Sobolev conjugate exponent of $q$, which would be $\frac{2q}{2-q}$ for $q\in[1,2)$.

\begin{lemma}\label{lem67f24a18}
	For every $q>1$ there is $\hat p>2$ such that, for $p>\hat p$,
	\begin{equation}\label{eq67f239dd}
		\frac{p^2}{(p-1)(p-2)} < q .
	\end{equation}
	For every $q>2$ there is $\hat p>2$ such that, for $p>\hat p$,
	\begin{equation}\label{eq67f239e4}
		\frac{ pq(p-2)}{p^2+q(p-2)} > q' = \frac{q}{q-1} .
	\end{equation}
	For every $q>4$ there is $\hat p>2$ such that, for $p>\hat p$,
	\begin{equation}\label{eq67f25035}
		\frac{ pq(p-2)}{p^2+q(p-2)} > \frac{2q}{q-2} > q' .
	\end{equation}
	For every $q>1$ there is $\hat p>2$ such that, for $p>\hat p$,
	if we set $s=\frac{ pq(p-2)}{p^2+q(p-2)}$ and $\alpha=p-1$, then 
	\begin{equation}\label{eq67f2503b}
		\frac{1}{p} + \frac{1}{s} + \frac{2}{\alpha} \le 1 .
	\end{equation}
\end{lemma}
\begin{proof}
	First, since 
	\[
	\lim_{p\to\infty} \frac{p^2}{(p-1)(p-2)} = 1 < q,
	\]
	then~\eqref{eq67f239dd} is satisfied for $p$ large enough.\\
	Second, notice that  
	\[
	\lim_{p\to\infty} \frac{ pq(p-2)}{p^2+q(p-2)} = q .
	\]
	So, if $q>2$, then $q> \frac{q}{q-1}$ and thus we obtain~\eqref{eq67f239e4} for $p$ large enough.
	If $q>4$, then $q>\frac{2q}{q-2} > \frac{q}{q-1}$ and we obtain~\eqref{eq67f25035} for $p$ large enough.\\
	Third, we have that
	\[
	\lim_{p\to\infty} \frac{1}{p} + \frac{1}{s} + \frac{2}{\alpha}
	= \frac1q .
	\]
	Since $q>1$, then~\eqref{eq67f2503b} is satisfied for $p$ large enough.
\end{proof}

\printbibliography
\end{document}